\documentclass[pdflatex,sn-mathphys-num]{sn-jnl}

\usepackage{graphicx}
\usepackage{amsmath,amssymb,amsfonts}
\usepackage{amsthm}
\usepackage{mathrsfs}
\usepackage{indentfirst}
\usepackage{url}

\theoremstyle{thmstyleone}%
\newtheorem{theorem}{Theorem}[section]

\newtheorem{lem}[theorem]{Lemma}

\newtheorem{prop}[theorem]{Proposition}

\theoremstyle{thmstylethree}%
\newtheorem{definition}[theorem]{Definition}

\theoremstyle{thmstyletwo}%

\newtheorem{obs}[theorem]{Remark}

\numberwithin{equation}{section}
\numberwithin{figure}{section}

\begin{document}

\title[A stochastic model for immune response]{A stochastic model for immune response with mutations and evolution: the non-spatial setting}

%% Authors and affiliations
\author[1]{\fnm{Carolina} \sur{Grejo}}%\email{carolina.grejo@sesisp.org.br}
\author[2]{\fnm{Fabio} \sur{Lopes}}%\email{f.marcellus@utem.cl}
\author[3]{\fnm{F\'abio} \sur{Machado}}%\email{fmachado@ime.usp.br}
\author[4]{\fnm{Alejandro} \sur{Rold\'an-Correa}}%\email{alejandro.roldan@udea.edu.co}

\affil[1]{\orgname{Faculdade SESI-SP de Educação}, \orgaddress{\city{Campinas}, \country{Brazil}}}
\affil[2]{\orgdiv{Departamento de Matem\'atica}, \orgname{Universidad Tecnol\'ogica Metropolitana}, \orgaddress{\country{Chile}}}
\affil[3]{\orgname{Institute of Mathematics and Statistics, Universidade de S\~ao Paulo}, \orgaddress{\city{S\~ao Paulo}, \country{Brazil}}}
\affil[4]{\orgname{Instituto de Matem\'aticas, Universidad de Antioquia}, \orgaddress{\city{Medell\'in}, \country{Colombia}}}

\abstract{We consider a stochastic model for a pathogen population in the presence of an immune response, in which pathogen types are partially ordered by ancestry and the immune system must eliminate ancestor types before it can eliminate their descendants. In this model, pathogens reproduce independently at rate $\lambda>0$ and, at each birth, a mutation occurs with probability $r\in(0,1]$, producing a novel type that is antigenically distinct and whose elimination by the immune system is delayed relative to its ancestors. We provide an explicit characterization of the survival--extinction phase transition and compute the expected total progeny in the subcritical regime.
We then extend the model by allowing mutations to be deleterious: conditional on mutation, with probability $p\in(0,1]$ the mutation is beneficial and with probability $1-p$ it is deleterious, producing a sterile offspring. For this extension, we obtain an explicit survival criterion in terms of $(\lambda,r,p)$ and identify parameter regimes in which survival is possible only for an intermediate range of mutation probabilities, reflecting the balance between immune escape and mutational load.}



\maketitle

	\section{Introduction}

Many pathogens evade host immunity by generating a continuing stream of antigenic variants through mutation and selection. 
This ``Red Queen'' coevolutionary pressure is prominent for RNA viruses (e.g.\ influenza and SARS-CoV-2) and for parasites exhibiting antigenic variation, and it motivates mathematical models across scales ranging from within-host dynamics to population-level epidemic evolution \cite{Cox,Garet,KucharskiAndreasenGog2016,MoraWalczak2023}. 
A central theme is a trade-off: higher mutation rates can create immune-escape variants, but they also incur a mutational load because most mutations are deleterious.

At the within-host scale, many classical models describe virus--immune interactions through predator--prey-type ordinary (or partial) differential equations, where immune activity is stimulated by pathogen abundance and in turn reduces pathogen growth; see, e.g., \cite{DeBoerPerelson1998,IwasaMichorNowak2004,NM2000,Sasaki1994}. 
In this setting, Sasaki \cite{Sasaki1994} studied antigenic drift/switching models in which only a small fraction of mutations produces successful immune-escape variants, and showed that this trade-off can select an \emph{intermediate} mutation rate that maximizes pathogen success (stationary pathogen density) under immune pressure. 
Related work by Sasaki and collaborators on multi-strain influenza dynamics with immune cross-reaction can be found in \cite{OmoriAdamsSasaki2010}, and recent evolutionary analyses at the epidemiological scale are in \cite{SasakiLionBoots2022}. 
From a complementary viewpoint, Eigen's quasispecies theory \cite{Eigen1971} and subsequent work on ``error thresholds'' and lethal mutagenesis \cite{BullSanjuaanWilke2007,SteinmeyerWilke2009} highlight that sufficiently high mutation rates can drive population collapse. 
More recently, traveling-wave/fitness-wave approaches have provided quantitative descriptions of antigenic advance, persistence, and extinction in multi-strain settings \cite{YanNeherShraiman2019,MarchiLassigMoraWalczak2021} and have been used to study the evolutionary stability of mutation rates under immune escape \cite{ChardesMazzoliniMoraWalczak2023}; see also \cite{MoraWalczak2023}.

In a different but complementary direction, Schinazi and Schweinsberg \cite{schinazi} introduced simple stochastic models for a pathogen population under an immune response, based on \emph{type-level} (global) killing rules. 
In their non-spatial Model~2, pathogens reproduce independently at rate $\lambda>0$ and, at each birth, a mutation occurs with probability $r\in(0,1]$, producing a new type that has never appeared before. 
Each type is removed after an independent exponential time of mean one, at which time \emph{all} individuals of that type are eliminated simultaneously. They show that pathogens can survive with positive probability if and only if $\lambda>1$. Hence, whether or not pathogens can survive depends only on the reproduction rate and not on the mutation rate. 
They also defined spatial analogues on $\mathbb{Z}^d$ \cite{schinazi}; in particular, Model~S2 exhibits a non-trivial dependence on $r$ in which, for $\lambda$ sufficiently large, survival with positive probability occurs for large mutation probabilities but extinction holds for small mutation probabilities. 
Liggett, Schinazi and Schweinsberg \cite{LSS2008} proved that Model~S2 on homogeneous trees has a behavior that  captures both features
from the model on $\mathbb{Z}^d$ (for an intermediate range of $\lambda$) and from the non-spatial Model~2 (for sufficiently large $\lambda$).

Schinazi and Schweinsberg \cite{schinazi} motivate global killing rules as a mathematical proxy for a ``central-command'' immune response acting in a coordinated way at the host level. 
In the HIV literature, for instance, Silvestri and Feinberg \cite{SilvestriFeinberg2003} argue that the pathogenesis of AIDS is driven in large part by chronic immune activation rather than by direct virus-induced cytopathicity, and subsequent work has further documented the systemic nature and consequences of immune activation and inflammation in chronic infection \cite{Brenchley2006,SodoraSilvestri2008,DeeksTracyDouek2013}. 
Our models adopt a similarly stylized perspective: we do not attempt to describe explicitly the dynamics of immune cells and their interactions with pathogens, but rather encode immune recognition and clearance at the level of pathogen types in a way that permits rigorous analysis.

The main modeling novelty of the present paper is to modify Model~2 by imposing an \emph{ancestral-order constraint}: pathogen types are partially ordered by ancestry and immune elimination respects this order, in the sense that the immune system must eliminate an ancestor type before it can eliminate its descendants. We call this model the \textit{beneficial-mutation model}. Here, the term ``beneficial-mutation'' refers to antigenic novelty (immune escape) rather than to an increased replication rate, since as in the models of Schinazi and Schweinsberg we assume that all types share the same intrinsic birth rate $\lambda$. Interestingly, the phase trasition diagram for survival obtained in the present paper (Theorem~\ref{TNS}) for the  beneficial-mutation model is qualitatively similar to the one in \cite{LSS2008} for homogeneous trees.
One possible interpretation of the ancestral-order constraint is immunodominance/resource allocation:
immune responses against concurrent antigenic variants can form strong hierarchies because responding
lymphocyte clones compete for limiting stimulatory resources. Such competition can suppress or delay responses to weaker or newly arising
variants until dominant responses contract or antigen availability is redistributed
\cite{Borghans1999,Kedl2000}. Moreover, in settings with antigenic variation, immune escape can shift the
immunodominant response over time \cite{NowakMay1995}, consistent with our interpretation that while an
abundant ancestral type persists, immune pressure is effectively focused on it, and sufficiently specific
control of newly arisen descendants builds up only after clearance of the ancestor.

Motivated by Sasaki's framework \cite{Sasaki1994} and by the error-threshold/lethal-mutagenesis literature \cite{Eigen1971,BullSanjuaanWilke2007,SteinmeyerWilke2009}, we then extend the model by allowing mutations to be either \emph{beneficial} or \emph{deleterious}. 
Conditional on mutation, with probability $p\in(0,1]$ the newborn is a new beneficial type, while with probability $1-p$ it is deleterious and produces a sterile offspring. 
For this mixed model we obtain a closed-form survival criterion in terms of $(\lambda,r,p)$ and identify parameter regimes in which survival is possible only for an \emph{intermediate} range of mutation probabilities $r$ (neither too low nor too high), reflecting the balance between the supply of immune-escape variants and mutational load; see Figure~\ref{fig:curvascriticas}. 
In this sense, our mixed model can be viewed as a rigorously analyzable minimalist stochastic framework that reproduces, via an alternative immune mechanism acting at the type level, qualitative behaviors previously reported by Sasaki \cite{Sasaki1994} and by other studies of within-host virus–immune dynamics modeled through complex predator--prey-type systems of ordinary (or partial) differential equations.

The models we introduce in this work are also naturally connected to stochastic processes with \emph{mass removal} or \emph{catastrophes}, such as forest-fire and mass-extinction dynamics \cite{JMR2016,L2011,MRS2015,schinazi4}, and to recent spatial models incorporating sterility effects \cite{Velasco}. 

\medskip
\noindent\textbf{Main results.}
For the beneficial-mutation model, we identify an explicit survival threshold (Theorem~\ref{TNS})
and compute the expected total progeny in the subcritical regime (Proposition~\ref{moments}). The proofs of these results combine
branching-process arguments adapted to the dependence induced by the immune ``type-killing'' rule with the analysis of a
recursive distributional equation for the total progeny of the process, leading to explicit fixed-point and
integral equations and, ultimately, closed-form critical curves. For the mixed model allowing deleterious mutations, we obtain a closed-form survival criterion (Theorem~\ref{TNSrp}),
derive subcritical moment formulas for the total number of pathogens ever born (Proposition~\ref{prop:moments-B-lrp}),
and describe the resulting phase diagram in $(\lambda,r,p)$ (Remark~\ref{obs:phase-diagram}). These results are obtained
via a thinning construction that relates the mixed model to an appropriate version of beneficial-mutation model with
effective parameters, allowing the derivation of the survival phase diagram and the moments of the mixed model with minimal additional work.

The paper focuses on the non-spatial setting. 
A companion paper treating the spatial model on $\mathbb{Z}^d$ is in preparation \cite{GLMRspatial}. 
A related spatial variant of the beneficial-mutation model has also been studied on homogeneous trees by Lopes and Rold\'an-Correa~\cite{Lopes_2025}.

\medskip
\noindent\textbf{Organization of the paper.} 
Section~2 introduces the  model with only beneficial mutations and presents the results on phase transition and moment formulas for the total progeny. 
Section~3 presents the extension adding deleterious mutations and describes the resulting phase diagram. 
Section~4 contains the proofs of the main results.

\section{Beneficial-mutation model }
	
	Next, we describe the dynamics of the beneficial-mutation model. At time zero, a single pathogen of type $1$ enters a host for the first time. The class of pathogens of type $1$ receives immediately a killing time given by an exponential random variable  of rate 1. During the killing time of type, each pathogen gives birth independently at rate $\lambda$. When a new pathogen is born, with probability $1-r$, it has the same type of its parent, and with probability $r$, a mutation occurs, and the new born pathogen has a new type that has never appeared previously in the system. Furthermore, we impose an \emph{ancestral-order} killing rule: a mutant type cannot be eliminated before its ancestor type has been eliminated. This can be interpreted as antigenic novelty (immune escape) along a lineage; note that all types have the same intrinsic birth rate $\lambda$, and the advantage of a beneficial mutation is solely in the delayed immune clearance. For this, when each type first appears in the system, it receives an independent random `clock'  with exponential distribution of rate 1. However, the `clock' of each type only starts ticking after its ancestor type has been eliminated. When the `clock' of a type rings, all pathogens of that type are eliminated simultaneously by the immune system.  Note that this is equivalent to say that when a pathogen is born it receives an independent `clock' distributed as a mixed random variable  which is 0 with probability $1-r$, and an exponential of rate 1, with probability $r$. For $r\in(0,1]$ and
    $\lambda > 0$, let's denote this process by $\mathcal{B}(\lambda, r)$.
	\begin{figure}[h!]
    	\centering	
		\includegraphics[scale=0.60]{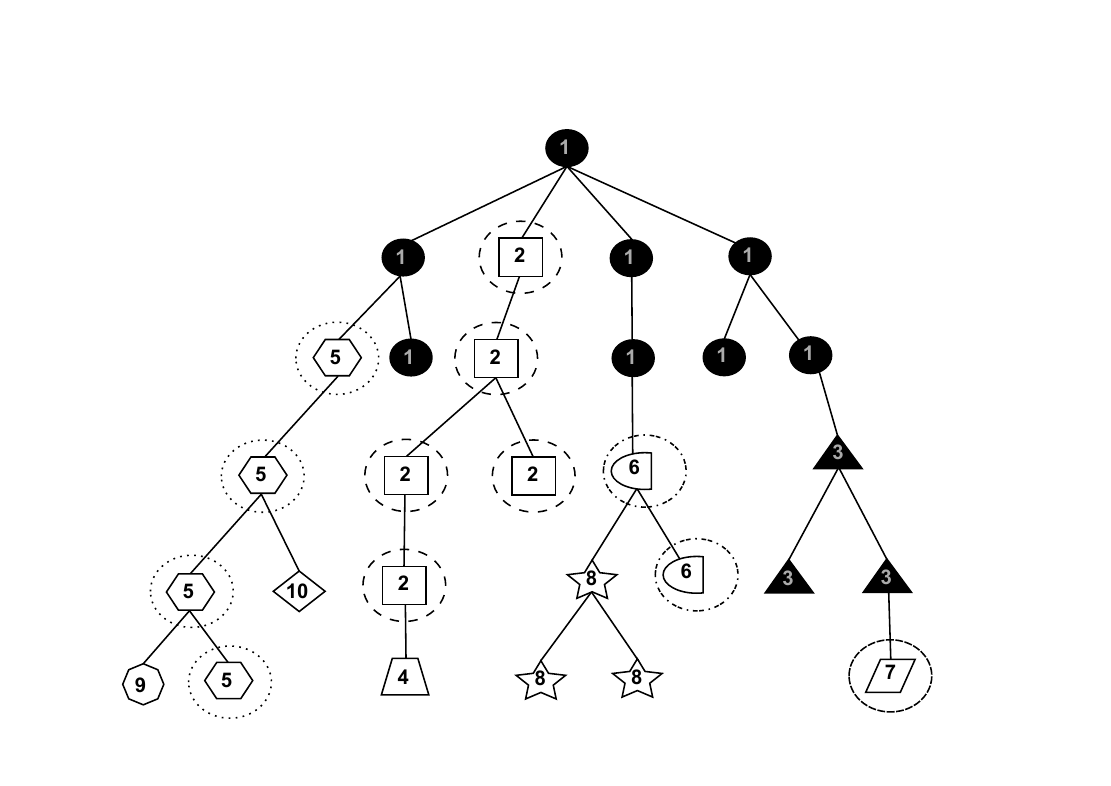}
		\caption{Illustration of the non-spatial process $\mathcal{B}(\lambda, r)$ with $0<r<1$. Pathogens of the same type share the same number and shape. They are numbered according to the order their type first entered into the system.   Pathogens alive are in white and
			pathogens already dead are in black. Pathogen types at risk are encircled. Pathogens of the same type die all together. For instance, if the next event to occur is the death of all pathogens of type 5. All pathogens of type 5 would change color, and pathogens of types 9 and 10 would become encircled.} \label{figuraarvore}
	\end{figure}

	The `clocks' of the pathogen types can be thought as the incremental time (or the killing time) that the immune system needs to recognize a new pathogen type after it has already managed to eliminate its ancestor type. Once a type is recognized the immune system is able to eliminate all pathogens of that type very effectively.

	\begin{definition} 
		
		We say that the process $\mathcal{B}(\lambda, r)$ survives if, with positive probability, there are pathogens alive for all $t>0$; otherwise, it dies out.
	\end{definition}
	Our main result concerning this model is the following. 
	
	\begin{theorem}\label{TNS} The process $\mathcal{B}(\lambda, r)$ dies out if and only if $$\lambda\leq(1+\sqrt{r})^{-2}.$$ 
	\end{theorem}
	By coupling arguments one can see that the survival probability of $\mathcal{B}(\lambda, r)$ is a non-decreasing function of $\lambda$ and also of $r$.
	From Theorem \ref{TNS}, for fixed $r\in (0,1]$, we obtain, 
	$\lambda_c(r):=(1+\sqrt{r})^{-2}$, the critical parameter (in $\lambda$) for the survival of $\mathcal{B}(\lambda, r)$. Note that if $\lambda\leq1/4$, then $\mathcal{B}(\lambda, r)$ dies out for all $r\in(0,1]$.  If $\lambda\geq1$, then $\mathcal{B}(\lambda, r)$ survives for all $r\in(0,1]$. However, if $1/4<\lambda<1$, then $\mathcal{B}(\lambda, r)$ dies out for $r\leq (1 - \sqrt{\lambda})^2/\lambda$, and survives for $r>(1 - \sqrt{\lambda})^2/\lambda$.   So, for $1/4<\lambda<1$, the process $\mathcal{B}(\lambda, r)$ presents a phase transition in $r$. See Figure~\ref{figura2}. This behavior is strikingly different from the Model 2 in \cite{schinazi}. In fact, this behavior is similar to the spatial version of model S2 on regular trees studied in Liggett \textit{et al.}~\cite{LSS2008}.
	
	Let $N$ denote the total progeny of the pathogen population in $\mathcal{B}(\lambda, r)$ i.e. the total number of pathogens entering the system including the initial pathogen. The sufficient condition in Theorem \ref{TNS} is a consequence of the next proposition.
	\begin{prop}
		\label{moments}
		For any fixed $r\in (0,1]$ and $\lambda>0$.
		\begin{itemize}
			\item [$(i)$] For $0<\lambda\leq (1+\sqrt{r})^{-2},$ 
			$$ \mathbb{E}[N] =\frac{2}{1-\lambda(1-r)+\sqrt{[1+\lambda (1-r)]^2 - 4\lambda}}.  $$
			In particular, at $\lambda=(1+\sqrt{r})^{-2}$, $ \mathbb{E}[N]= \frac{1+\sqrt{r}}{\sqrt{r}}$.\\
			\item[$(ii)$] For $\lambda>(1+\sqrt{r})^{-2},$  $\mathbb{E}[N]=\infty$.
		\end{itemize}
	\end{prop}
	We note that  $\mathbb{E}[N]<\infty$ implies that $N<\infty$ almost surely. Consequently, the pathogens die out with probability 1 for $\lambda\leq(1+\sqrt{r})^{-2}$. For the necessary condition in Theorem \ref{TNS},  we will show that the  $\mathcal{B}(\lambda, r)$ survives for $\lambda> (1+\sqrt{r})^{-2}$. 
		\begin{figure}[h]
	\centering		\includegraphics[scale=0.65]{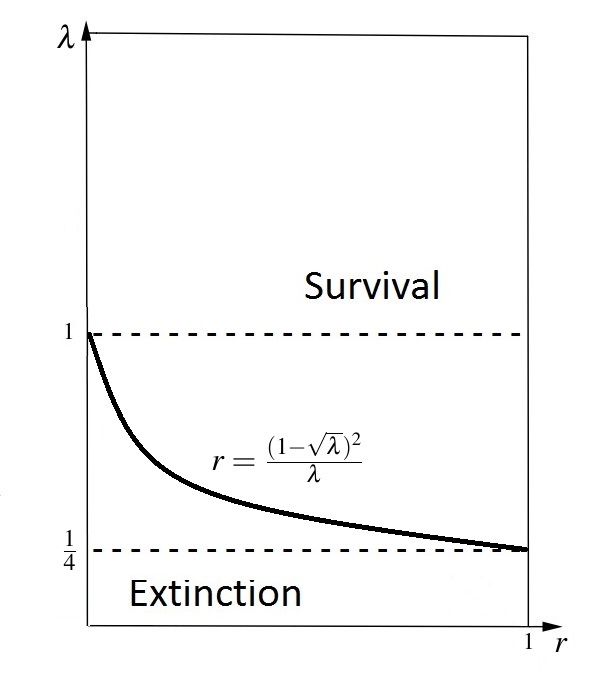}
	\caption{The process $\mathcal{B}(\lambda,r)$ survives with positive probability if and only if $(\lambda,r)$ lies above the critical curve $r=(1-\sqrt{\lambda})^{2}/\lambda$.}
\label{figura2}
		\end{figure}

	\subsection{Connections with the birth and assassination process}%\hfill
	
    Aldous and Krebs~\cite{aldous} introduced the \textit{birth and assassination process} to investigate the scaling limit of a certain queuing system with blocking, presented in Tsitsiklis~\textit{et al.}~\cite{TPH}. Aldous and Krebs' model is a variant of a continuous-time branching process in which each  particle reproduces independently at rate $\lambda$ and has an independent random killing time 
	whose `clock' only starts ticking after its parent is dead. Recently, this model has attracted considerable attention, in particular because it exhibits an interesting heavy-tail phenomena and can be related to different SIR models and branching random walks \cite{AB,B2008,B2014,GLMR,K}.
	
	The birth and assassination process with killing times given by independent exponential random variables with mean 1 is equivalent to  $\mathcal{B}(\lambda, 1)$. The process $\mathcal{B}(\lambda, r)$ can be seen as the a variation of the birth and assassination process in which each particle is removed together with its parent with probability $1-r$, and with probability $r$, it has its independent `clock' that starts counting down its lifetime after its parent has been removed from the system. For the $\mathcal{B}(\lambda, 1)$ it is known that the processes dies out if and only if $\lambda\leq 1/4$,  see \cite{aldous, B2008}. Our Theorem \ref{TNS} extends this result for  $\mathcal{B}(\lambda, r)$ with $r\in (0,1]$.

\section{Model adding deleterious mutations}

In this section we extend the model of the previous section by allowing mutations to be either beneficial or deleterious.

Fix $\lambda>0$, $r\in(0,1]$ and $p\in(0,1]$. As before, each pathogen gives birth independently at rate $\lambda$. When a new pathogen is born, it is
\begin{itemize}
    \item a \emph{copy} of its parent (same type) with probability $1-r$;
    \item a \emph{mutant} with probability $r$. Conditional on mutation, the mutation is
    \begin{itemize}
        \item \emph{beneficial} with probability $p$, in which case the newborn pathogen is a new type that has never appeared in the system and is antigenically novel relative to its parent type (and is therefore cleared no earlier than its ancestor lineages);
        \item \emph{deleterious} with probability $1-p$, in which case the newborn pathogen is \emph{sterile}, i.e.\ it produces no offspring.
    \end{itemize}
\end{itemize}
The immune response is the same as in the original model: when a \textit{mutant} type first appears, it receives an independent exponential clock of rate $1$ which starts ticking only after its ancestor type has been eliminated; when the clock rings, all pathogens of that type are eliminated simultaneously. Sterile mutant types do not produce descendants but they receive immune clocks and are eventually eliminated. Since a fraction of reproduction events are “wasted” on offspring that cannot reproduce, the survival--extinction phase transition may depend on $(r,p)$. We denote this extended process by $\mathcal{B}(\lambda,r,p)$.

\begin{definition}
We say that the process $\mathcal{B}(\lambda,r,p)$ \emph{survives} if, with positive probability, there are pathogens alive for all $t>0$; otherwise it \emph{dies out}.
\end{definition}

Our main result for this extension provides an explicit characterization of the phase transition.

\begin{theorem}\label{TNSrp}
Let $\lambda>0$, $r\in(0,1]$ and $p\in(0,1]$. The process $\mathcal{B}(\lambda,r,p)$ dies out almost surely if and only if
\begin{equation}\label{eq:theta-closed}
\lambda\left(\sqrt{1-r+rp}+\sqrt{rp}\right)^2\le 1.
\end{equation}
\end{theorem}

\begin{obs}[Phase diagram for fixed $p$]\label{obs:phase-diagram}
By standard coupling arguments, the survival probability of $\mathcal{B}(\lambda,r,p)$ is a non-decreasing function of $\lambda$. 
Define the critical value in $\lambda$ by
\[
\lambda_c(r,p):=\big(\sqrt{1-r+rp}+\sqrt{rp}\big)^{-2},
\]
so that, by Theorem~\ref{TNSrp}, the process survives if and only if $\lambda>\lambda_c(r,p)$.

Fix $p\in(0,1]$. 
The qualitative shape of the critical curve $r\mapsto \lambda_c(r,p)$ depends on whether $p\ge \tfrac12$ or $p<\tfrac12$.

\smallskip
\noindent\textbf{(i) Case $p\ge \tfrac12$.} 
In this regime $r\mapsto \lambda_c(r,p)$ is strictly decreasing on $(0,1]$, with
\[
\lambda_c(0^+,p)=1
\qquad\text{and}\qquad
\lambda_c(1,p)=\frac{1}{4p}.
\]
Consequently, if $\lambda\le \tfrac{1}{4p}$ then $\mathcal{B}(\lambda,r,p)$ dies out for all $r\in(0,1]$, while if $\lambda\ge 1$ it survives for all $r\in(0,1]$.
For intermediate values $\tfrac{1}{4p}<\lambda<1$, there is a unique threshold $r_-(\lambda,p)\in(0,1)$ such that survival holds if and only if $r>r_-(\lambda,p)$, where
\[
r_-(\lambda,p)=\frac{\big(\sqrt{p}-\sqrt{\lambda+p-1}\big)^2}{\lambda}.
\]

\smallskip
\noindent\textbf{(ii) Case $p< \tfrac12$.} 
In this regime the critical curve is \emph{non-monotone}: it is strictly decreasing on $(0,r^\ast)$ and strictly increasing on $(r^\ast,1]$, where
\[
r^\ast=\frac{p}{1-p}\in(0,1)
\]
is the unique minimizer and $\lambda_c(r^\ast,p)=1-p$ is the minimum value.
In particular, if $\lambda\le 1-p$ then extinction holds for all $r\in(0,1]$.
When $1-p<\lambda<\min\{1,\frac{1}{4p}\}$, the equation $\lambda_c(r,p)=\lambda$ has two solutions $0<r_-(\lambda,p)<r_+(\lambda,p)<1$ and
\[
\mathcal{B}(\lambda,r,p)\ \text{survives}\quad\Longleftrightarrow\quad r\in\big(r_-(\lambda,p),\,r_+(\lambda,p)\big),
\]
 where 
\[
r_\pm(\lambda,p)
=\frac{\big(\sqrt{p}\pm\sqrt{\lambda+p-1}\big)^2}{\lambda}
=\frac{\lambda+2p-1\ \pm\ 2\sqrt{p(\lambda+p-1)}}{\lambda}.
\]
 This gives the \emph{intermediate-mutation window} in $r$. If $\frac{1}{4}<p<\frac{1}{2}$ and $\frac{1}{4p}<\lambda<1$ then the survival holds if and only if $r>r_{-}(\lambda,p)$.
If $p<\tfrac14$ and $1\le \lambda<\tfrac{1}{4p}$, then $\lambda>\lambda_c(r,p)$ holds for small mutation probabilities but fails for large ones: there is a unique $r_+(\lambda,p)\in(r^\ast,1)$ such that survival holds if and only if $r<r_+(\lambda,p)$.
Finally, if $\lambda\ge \max\{1,\tfrac{1}{4p}\}$ then survival holds for all $r\in(0,1]$.

Figure~\ref{fig:curvascriticas} illustrates the critical curves $r\mapsto \lambda_c(r,p)$ for several values of $p$.
\end{obs}

\begin{figure}[h!]
\centering
\includegraphics[width=0.70\textwidth]{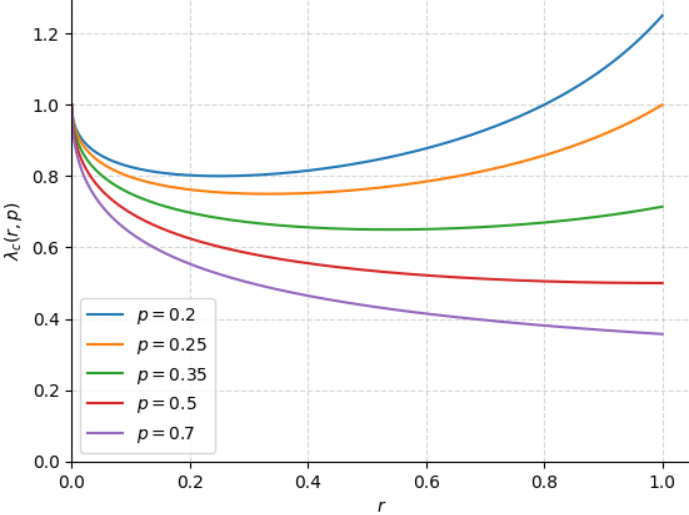}
\caption{Critical curves $r\mapsto \lambda_c(r,p)$ for different values of $p$. For fixed $p$, the region above the curve corresponds to survival ($\lambda>\lambda_c(r,p)$), while the region on and below corresponds to almost sure extinction ($\lambda\le \lambda_c(r,p)$).}
\label{fig:curvascriticas}
\end{figure}

Allowing mutations to be beneficial only with probability $p$ makes explicit the trade-off that mutation both \emph{creates immune-escape variants} (needed for persistence) and \emph{wastes reproduction on unfit offspring} (mutational load). 
As a consequence, the survival criterion in Theorem~\ref{TNSrp} can generate an \emph{intermediate-mutation} regime in which persistence is possible only for a range of mutation probabilities $r$: too little mutation produces too few successful new types, while too much mutation generates too many deleterious lineages.
This behavior echoes Sasaki's antigenic drift/switching models \cite{Sasaki1994} and classical error-threshold/lethal-mutagenesis mechanisms \cite{Eigen1971,BullSanjuaanWilke2007,SteinmeyerWilke2009}, and it is also consistent with recent traveling-wave descriptions of immune escape and mutation-rate trade-offs \cite{YanNeherShraiman2019,MarchiLassigMoraWalczak2021,ChardesMazzoliniMoraWalczak2023,MoraWalczak2023}.

Next, we analyze the total progeny of the process $\mathcal B(\lambda,r,p)$.  Let $N_{\mathrm v}$ be the total number of \emph{viable} individuals ever born, including the initial ancestor and all
copies and beneficial mutants, but excluding sterile mutants. Let $S$ be
the total number of sterile individuals ever born. Define
\[
N^{(p)} := N_{\mathrm v}+S,
\]
the total number of pathogens ever born.

\begin{prop}
\label{prop:moments-B-lrp}
Let $\lambda>0$, $r\in(0,1]$ and $p\in(0,1]$, and consider the process $\mathcal B(\lambda,r,p)$.

Set
\[
\Delta(\lambda,r,p):=\bigl(1+\lambda(1-r)\bigr)^2-4\lambda(1-r+rp),
\qquad
\lambda_c(r,p):=\frac{1}{\bigl(\sqrt{1-r+rp}+\sqrt{rp}\bigr)^2}.
\]
Then:
\begin{enumerate}
\item If $0<\lambda\le \lambda_c(r,p)$, one has $\mathbb{E}[N^{(p)}]<\infty$ and
\begin{eqnarray*}
   \mathbb{E}[N_{\mathrm v}]&=&
\frac{2}{\,1-\lambda(1-r)+\sqrt{(1+\lambda(1-r))^2-4\lambda(1-r+rp)}\,}. \\
\mathbb{E}[S]&=&\frac{r(1-p)}{1-r(1-p)}\,(\mathbb{E}[N_{\mathrm v}]-1).\\
\mathbb{E}[N^{(p)}]&=& \frac{1}{1-r(1-p)}\left(
\frac{2}{\,1-\lambda(1-r)+\sqrt{\Delta(\lambda,r,p)}\,}
-r(1-p)
\right).
\end{eqnarray*}
In particular, at criticality $\lambda=\lambda_c(r,p)$ (so that $\Delta(\lambda,r,p)=0$),
\[
\mathbb{E}[N^{(p)}]
=
\frac{1}{1-r(1-p)}\left(
\frac{1}{1-\lambda(1-r)}-r(1-p)
\right).
\]
\item If $\lambda>\lambda_c(r,p)$, then $\mathbb{E}[N^{(p)}]=\mathbb{E}[N_{\mathrm v}]=\mathbb{E}[S]=\infty$.
\end{enumerate}
\end{prop}

\section{Proofs}
	
	\subsection{Beneficial-mutation model}
	We begin by introducing some notations and a formal definition of model $\mathcal{B}(\lambda,r).$
	Let $\mathcal{N}=\displaystyle\cup_{n=0}^{\infty}{\mathbb{N}}^n$, where $\mathbb{N}$ is the set of positive integers.
	Each $\textbf{n}\in \mathcal{N}$ is a finite $n-$tuple of positive integers, with ${\mathbb{N}}^0=\emptyset$.
	
	The initial pathogen in the system is identified as $\emptyset$, and declared to be the $0$-th generation. We note that, each descendant can be mapped to an element of $\mathcal{N}$ as follows. For %$\textbf{n}^{'}$, $\textbf{n}\in \mathcal{N}$, such that
	$\textbf{n}^{'}=(n_1,...,n_{k-1})$ and $\textbf{n}=(n_1,...,n_{k-1}, n_k)$ in $\mathcal{N}$, %for some $k\geq 1$* 
	while $\textbf{n}$ is the ${n_k}$-th child of $\textbf{n}^{'}$, both are descendants of
	$\emptyset$, being $\textbf{n}^{'}$ at the $(k-1)$-th generation and $\textbf{n}$ at the $k$-th generation. 
	Now, for $\textbf{n}\in \mathcal{N}$, let $k(\textbf{n})$ denote the number of coordinates in $\textbf{n}$.  We set $k(\emptyset)=0$.
	
	Next, for $r\in(0,1]$ and $\lambda >0$, we formally introduce the  process $\mathcal{B}(\lambda,r).$  For $\textbf{n}\in\mathcal{N}$, let $\{X_\textbf{n}\}$ be a family of independent Poisson process with common arrival rate $\lambda$, and $\{K_\textbf{n}\}$ a family of independent exponential random variables of rate 1. 
	
	At time $0$, the initial pathogen $\emptyset$ produces offspring until time $D_{\emptyset}=K_{\emptyset}$. Each of its offspring is independently tagged as \textit{separated} with probability $r$. Every pathogen begins producing offspring from the moment of its birth, following its associated Poisson process. At time $D_{\emptyset}$, the initial pathogen $\emptyset$, all its children that are not tagged as \textit{separated} and all descendants that do not descend of descendants of  $\emptyset$ tagged as \textit{separated}, are removed. Let $k>0$, $\textbf{n}^{'}=(n_1,\dots,n_{k-1})$ and $\textbf{n}=(n_1,\dots,n_{k-1},n_k)$. When a pathogen $\textbf{n}^{'}$ is removed from the system at a certain time $D_{\textbf{n}^{'}}$, all its children that are not tagged as \textit{separated} and all descendants that do not descend of descendants of $\textbf{n}^{'}$ tagged as \textit{separated} are also removed. If $\textbf{n}$ is tagged as \textit{separated}, then it continues to produce offspring until time $D_{\textbf{n}}=D_{\textbf{n}^{'}}+K_{\textbf{n}}$. At time $D_{\textbf{n}}$,  $\textbf{n}$ is removed from the system, together with all its children that are not tagged as \textit{separated} and all descendants that do not descend of descendants of $\textbf{n}$ tagged as \textit{separated}.

	In order to prove the necessary condition in Theorem \ref{TNS}, it is convenient to define  $\mathcal{B}^*(\lambda,r)$, an alternative version of $\mathcal{B}(\lambda,r)$ in which the initial pathogen also receives a killing time given by
	a mixed random variable  which is 0 with probability $1-r$, and an exponential of rate 1, with probability $r$.

	Let $q$ and $q^*$ denote the probability of extinction of $\mathcal{B}(\lambda,r)$ and $\mathcal{B}^{*}(\lambda,r)$, respectively. Note that conditioning on the killing time of the initial pathogen in $\mathcal{B}^{*}(\lambda,r)$ we have that 
	$$ q^*=(1-r)\cdot 1 + r\cdot q.  $$
	It is easy to check that $q^*=1$ if and only if $q=1$. Finally, we present a lemma that we need to study the survival of $\mathcal{B}^{*}(\lambda,r)$. 
	\begin{lem}
		\label{lemaald} Consider $\mathcal{B}^{*}(\lambda,r)$. If $\lambda>(1+\sqrt{r})^{-2}$,  then there exists $k\in\mathbb{N}$ sufficiently large such that  
		$$\sum_{k(\textbf{n})=k}\mathbb{P}(\text{\textbf{n} is born})>1/r.$$
	\end{lem}
	\begin{proof}
		
		Let $\{B_{\nu,i}\},$ $\nu,i=1,2,\ldots,$ be independent exponential random variables with parameter $\lambda.$ Let $\{K'_i\}_{i\geq 1}$ be a sequence of independent random variables distributed as a mixed random variable which is zero with probability $1-r,$ and an exponential random variable of rate 1 with probability $r$.\\
		In order to the pathogen $\textbf{n}=(n_1,\ldots,n_k)\in\mathcal{N}$ to be born in $\mathcal{B}^{*}(\lambda,r)$, its ancestor in the $(i-1)-$th generation must have had at least $n_i$ children for each $i \in \{1, \dots, k\}$.  So, 
		$$\mathbb{P}(\text{\textbf{n} is born})=
		\mathbb{P}\left[\sum_{i=1}^j\sum_{\nu=1}^{n_i}B_{\nu,i}<\sum_{i=1}^j K'_i \, , \, j=1,2,\ldots,k(\textbf{n})\right].$$
		The moment generating function of the variables $K'$ is given by
		$\phi(u)=1-r+\frac{r}{1-u}$. Furthermore, the condition $\lambda>(1+\sqrt{r})^{-2}$ is equivalent to $\min_{u>0}\left\{\frac{\lambda}{u}\phi( u)\right\}>1$.  
		Let $\epsilon>0$ be such that $\min_{u>0}\left\{\frac{\lambda}{u}\phi(u)\right\}=1+\epsilon$.  
		Analogously to the proof of Lemma 2 \cite{aldous}, we  obtain that for sufficiently large $k$
		$$\sum_{k(\textbf{n})=k}\mathbb{P}(\text{\textbf{n} is born})>(1+\epsilon/2)^k.$$
	\end{proof}
	
	\begin{proof}[Proof Theorem \ref{TNS} (Necessary condition).]
		In order to show that the process $\mathcal{B}^{*}(\lambda,r)$ survives with positive probability when $\lambda>(1+\sqrt{r})^{-2}$, we define an auxiliary process which has to be slightly different from the one in \cite{aldous} due to the additional dependencies in our model. Let $k\in \mathbb{N}$, we say that  a pathogen in $\mathcal{B}^{*}(\lambda,r)$ is an \textit{special particle} if it is the initial pathogen $\emptyset$, or if it is a new mutation in generation $nk$ that is descended from a special particle in generation $(n-1)k$ without ancestors alive.
		Let $Z_n$ denote the number of special particles in the generation $nk$ of $\mathcal{B}^{*}(\lambda,r)$. Observe that the numbers of special particles created by different special particles are independent and identically distributed. This follows from the fact that only descendants born after their special ancestor has no living ancestors can be special particles. Thus, the process $(Z_n)_{n\geq0}$ defines a Galton-Watson branching process with $Z_0=1$. For each $\textbf{n}$, the probability that  the pathogen $\textbf{n}$ is born and is a mutation is $r\mathbb{P}(\text{\textbf{n} is born})$. 
		Thus, the expected number of special particles in generation $k$ in $\mathcal{B}^{*}(\lambda,r)$ is given by $$\mathbb{E}[Z_1]=\sum_{k(\textbf{n})=k}r\mathbb{P}(\text{\textbf{n} is born}).$$
		Lemma \ref{lemaald} implies that the last quantity is greater than 1 for some $k$ sufficiently large. So, the process $(Z_n)_{n\geq0}$ is a supercritical branching process.
		As a consequence $\mathcal{B}^{*}(\lambda, r)$ and $\mathcal{B}(\lambda, r)$ survive with positive probability.
	\end{proof}

	Next, we prove Proposition \ref{moments} which implies the sufficient condition in Theorem~\ref{TNS}. Our proof follows the approach proposed by Bordenave to study  the moments of the total progeny for the birth and assassination process \cite{B2008} and the chase-escape process on trees \cite{B2014}. The main novelty in the proof of proposition is the consideration of a new term which is related to pathogens which are copies. The key condition to obtain such extension is the analysis presented in  Lemma \ref{key}.
	
	Let $Y(t)$ denote the total number of pathogens that enter the system (including the initial pathogen) given that the initial pathogen dies at time $t$. Note that $\mathbb{E}[Y(0)]=1$. The following lemma from \cite[Lemma 1]{B2008} for $\mathcal{B}(\lambda, 1)$ also holds for our generalization $\mathcal{B}(\lambda, r)$, $r\in (0,1]$.
	\begin{lem}(Bordenave)
		\label{Lemma1}
		Let $t>0$ and $u>0$, if $\mathbb{E}[N^u] <\infty$ then $\mathbb{E}[{Y(t)}^u]<\infty$.
	\end{lem}
	As argued in \cite{B2008}, if $\mathbb{E}[N^u]<\infty$ for some $u>0$, then Fubini's theorem implies that $\mathbb{E}[N^u]=\int_0^{\infty} \mathbb{E}[Y^{u}(t)] e^{-t}dt$ and therefore $\mathbb{E}[Y^u(t)]<\infty$ for almost all $t\geq0$. This can be extended for all $t\geq0$ since for all $0<s\leq t$,   $Y(s)$ is stochastically dominated by $Y(t)$.\\
	
	Now we define $X(t)$ to be the total number of pathogens that enter in the system given that the initial pathogen $\emptyset$ cannot die before time $t$. By definition, the killing time of the initial pathogen $\emptyset$, denoted by $K_{\emptyset}$, is exponentially distributed with mean 1 and is independent of $Y$, consequently, $N\stackrel{d}{=} X(0) \stackrel{d}{=} Y(K_{\emptyset})$, where the symbol $\stackrel{d}{=}$ stands for equality in distribution. Moreover, the lack of memory of the exponential random variables implies that $X(t)\stackrel{d}{=} Y(t+K_{\emptyset})$. From these observations, we can derive the following equality in distribution
	$$ Y(t) \stackrel{d}{=} 1 + \sum_{i: \xi_{i}^{(1)}\leq t} X_{i}(t-\xi_{i}^{(1)}) +  \sum_{i: \xi_{i}^{(2)}\leq t} Y_{i}(t-\xi_{i}^{(2)})  \stackrel{d}{=} 1 + \sum_{i: \xi_{i}^{(1)}\leq t} X_{i}(\xi_{i}^{(1)}) +  \sum_{i: \xi_{i}^{(2)}\leq t} Y_{i}(t-\xi_{i}^{(2)}),$$
	where $\Phi^{(1)}=\{\xi_{i}^{(1)}\}$ and $\Phi^{(2)}=\{\xi_{i}^{(2)}\}$ are two independent Poisson point process of intensity $\lambda r$ and $\lambda (1-r)$, respectively, $\{X_i\}_{i\geq1}$ and  $\{Y_i\}_{i\geq1}$ are independent copies of $X$ and $Y$, respectively. We observe that $\Phi^{(1)}$ is related to the children tagged as \textit{separated}, and $\Phi^{(2)}$ otherwise.  Since all the variables are nonnegative, there is no issue with the case in which $Y(t)=\infty$. From this last equality, we can obtain the following recursive distributional equation (RDE) for the random function $Y$:
	\begin{align} Y(t) \stackrel{d}{=}  1 + \sum_{i: \xi_{i}^{(1)}\leq t} Y_{i}(\xi_{i}^{(1)}+K_i) +  \sum_{i: \xi_{i}^{(2)}\leq t} Y_{i}(t-\xi_{i}^{(2)}), \label{RDE}  \end{align}
	where $Y_i$ are independent copies of $Y$ and $K_{i}$ are independent exponential random variables with mean 1. The main difference between (\ref{RDE}) and the RDE obtained in \cite[Equation (3)]{B2008} for $\mathcal{B}(\lambda, 1)$ is the extra term due to $\Phi^{(2)}$.\\

	The idea of the proof in \cite{B2008,B2014} can be described as follows. Assuming that $\mathbb{E}[N]<\infty$, we find a good candidate for $\mathbb{E}[Y(t)]$ and show that necessarily $\lambda\in (0,(1+\sqrt{r})^{-2})$.  Subsequently,  we show that  $\mathbb{E}[N]<\infty$ for $\lambda\in (0,(1+\sqrt{r})^{-2}]$, and compute  $\mathbb{E}[N]=\int_0^\infty \mathbb{E}[Y(t)] e^{-t}dt$.\\
	
	Assuming that  $\mathbb{E}[N]<\infty$, we take the expectation in (\ref{RDE}) we obtain by Campbell's formula \cite{Baccelli} that,
	$$ \mathbb{E}[Y(t)] = 1 + \lambda r \int_{0}^{t} \int_{0}^{\infty} \mathbb{E}[Y(x+s)]e^{-s}ds dx +  \lambda (1-r)  \int_{0}^{t} \mathbb{E}[Y(t-x)] dx. $$
	Let $f_1(t)=\mathbb{E}[Y(t)]$, so it satisfies the following integral equation, for all $t\geq0$,
	\begin{align}\label{A}
		f_1(t) = 1 + \lambda r \int_{0}^{t} e^x \int_{x}^{\infty} f_1(s)e^{-s}ds dx +  \lambda (1-r)  \int_{0}^{t} f_1(t-x) dx. 
	\end{align}
	Taking the derivative of (\ref{A}) once and multiplying it by $e^{-t}$, and taking the derivative a second time and multiplying by $e^t$, we obtain  that $f_1$ satisfies the following linear ordinary differential equation of second order
	\begin{align}\label{B}
		x'' -[1+\lambda(1-r)]x' + \lambda x= 0, \end{align}
	with initial condition $x(0)=1$ since $\mathbb{E}[Y(0)]=1$. \\
	The discriminant of the polynomial $X^2-[1+\lambda(1-r)]X+\lambda=0$ is given by $\Delta=[1+\lambda(1-r)]^2 - 4\lambda$. It is easy to verify that $\Delta=0$ at $\lambda_1=(1+\sqrt{r})^{-2}$ and $\lambda_2=(1-\sqrt{r})^{-2}$, and strictly negative within the interval $(\lambda_1, \lambda_2)$.
	So, for $\lambda\in (\lambda_1, \lambda_2)$, the solutions of (\ref{B}) are of the form 
	$$ x_a(t)=e^{[1+\lambda(1-r)]t\slash 2}\left( \cos(t\sqrt{4\lambda -[1+\lambda(1-r)]^2}) + a\sin(t\sqrt{4\lambda -[1+\lambda(1-r)]^2}) \right), $$
	for some constant $a$. Since $f_1(t)$ must be positive for all $t>0$, this leads to a contradiction and $\mathbb{E}[N]=\infty$. Since $N$ is non-decreasing in $\lambda$, the same conclusion holds for all $\lambda>\lambda_1$ as well. We have just proved Proposition \ref{moments} $(ii)$.\\
	
	Next we only consider $0<\lambda\leq \lambda_1$, in which case, $\Delta\geq 0$ and the roots of the polynomial $X^2-[1+\lambda(1-r)]X+\lambda=0$ are real numbers, $\alpha$ and $\beta$, given by
	$$  \alpha= \frac{1+\lambda(1-r)- \sqrt{\Delta}}{2} \hbox{ and } \beta= \frac{1+\lambda(1-r)+ \sqrt{\Delta}}{2}. $$
	
	When $0<\lambda<\lambda_1$, the solutions of (\ref{B}) are of the form 
	$$ x_a(t)= (1-a)e^{\alpha t } + a e^{\beta t}, $$
	for some constant $a$.\\
	While, if $\lambda=\lambda_1$, the solutions of (\ref{B}) are of the form
	$$x_{a}(t)= (a t + 1) e^{\alpha t}=(a t + 1) e^{t\slash (1+\sqrt{r})}, $$
	for some constant $a$. \\
	It is easy to see that, the functions $x_a(t)$ with $a\geq0$ are the non-negative solutions of (\ref{A}), and  therefore, possible candidates for $\mathbb{E}[Y(t)]$. \\
	
	In order to finish the proof of Proposition \ref{moments}, 
	it remains to show that, for $0<\lambda \leq \lambda_1$, $\mathbb{E} [N]<\infty$ and that $f_1(t)=e^{\alpha t}$ i.e. $x_a(t)$ with $a=0$. Consequently, 
	$$\mathbb{E}[N]=\int_0^{\infty} f_1(t)e^{-t} dt=\frac{1}{1-\alpha}= \frac{2}{1-\lambda(1-r)+\sqrt{[1+\lambda(1-r)]^2 - 4\lambda}},$$ which coincides with the expression in Proposition \ref{moments}. \\
	
	Next we proceed to show that, if $0<\lambda \leq \lambda_1$, then  $\mathbb{E} [N]<\infty$ and that $f_1(t)=x_0(t)=e^{\alpha t}$.\\
	
	In order to proceed we need a lemma analogous to \cite[Lemma 5.2]{B2014}. First, for $0<\lambda\leq \lambda_1$, we define
	$$ \bar{\gamma}(\lambda, r) = \frac{1+\lambda(1-r) + \Delta}{1+\lambda(1-r) - \Delta}= \frac{\beta}{\alpha}, $$
	as before $\alpha$ and $\beta$ are the roots of the polynomial $X^2-[1+\lambda(1-r)]X +\lambda=0$.\\
	The next lemma collects two properties of $ \bar{\gamma}(\lambda, r)$ that are key in what follows.
	
	\begin{lem}
    For fixed $r\in (0,1]$ and  $0<\lambda<\lambda_1$. \\
		\label{key}
		If $1<u<\bar{\gamma}(\lambda, r)$, then it holds that,
		\begin{itemize}
			\item [$(i)$] $u\alpha<1$,
			\item [$(ii)$] $u\alpha[1+\lambda(1-r) - u\alpha]>\lambda$
		\end{itemize}
	\end{lem}
	\begin{proof}[Proof of Lemma \ref{key}]
		We start with the claim in  $(i)$. By definition, $u\alpha<\beta$. We observe that, for $r\in (0,1]$ and $\lambda=0$, $\beta=1$. As shown below, the derivative of $\beta$ with respect to $\lambda$ is strictly negative for $\lambda\in (0,\lambda_1)$. So, $\beta<1$ for $\lambda\in (0,\lambda_1)$.
		\begin{align*} \frac{d \beta}{d \lambda} %&=\frac{1}{2}\left( 
			% (1-r) + \frac{(1-r)+\lambda (1-r)^2 - 2}{\sqrt{\Delta}}  \right)\\
			&=\frac{1}{\sqrt{\Delta}}\left( \frac{(1-r)[1+\lambda(1-r)+ \sqrt{\Delta}] -2 }{2}
			\right)\\
			&<\frac{1}{\sqrt{\Delta}} \left( \frac{[1+\lambda(1-r)+ \sqrt{\Delta}] -2 }{2}
			\right)< -\frac{\lambda r}{2\sqrt{\Delta}}.
		\end{align*}
		The last inequality follows from  $\Delta=[1+\lambda(1-r)]^2 - 4\lambda <  (1-\lambda)^2$.\\
		Next we show that the claim in $(ii)$ holds. We observe that,  $1<u<\bar{\gamma}(\lambda, r)$ is equivalent to $u=1+\epsilon$, $\epsilon\in \left(0, \frac{\sqrt{\Delta}}{\alpha}\right)$, and that $\alpha[1+\lambda(1-r)-\alpha]=\lambda$. So,
		\begin{align*}
			(1+\epsilon)\alpha[1+\lambda(1-r)-(1+\epsilon)\alpha] &= \lambda + \epsilon \alpha[1+\lambda(1-r) -2\alpha -\epsilon\alpha]\\
			&= \lambda + \epsilon \alpha[\sqrt{\Delta}-\epsilon \alpha]>\lambda.
		\end{align*}
		The last step follows from the fact that, $\sqrt{\Delta}-\epsilon \alpha> 0$, for $\epsilon\in \left(0, \frac{\sqrt{\Delta}}{\alpha}\right)$.
	\end{proof}
	
	For $1<u<\bar{\gamma}(\lambda, r)$, we define $\mathcal{H}_u$ as the set of measurable functions $h: [0,\infty) \rightarrow [0,\infty)$ such that $h$ is non-decreasing and $\sup_{t\geq0} h(t)e^{-u \alpha t}<\infty$. Let $L>0$, we define the following mapping from $\mathcal{H}_u$ to $\mathcal{H}_u$,
	$$ \Psi_u: h \rightarrow L e^{u\alpha t} + \lambda r \int_{0}^{t} e^x \int_{x}^{\infty} h(s) e^{-s} ds dx + \lambda(1-r) \int_{0}^{t} h(s) ds.$$
	
	We note that for $r=1$, $\Psi_u$ coincides with the mapping introduced in \cite{B2008} to study the total progeny of $\mathcal{B}(\lambda, 1)$. As in \cite{B2008}, the fact that $\Psi_u$ is a mapping from $\mathcal{H}_u$ to $\mathcal{H}_u$ can be easily verified using Lemma \ref{key} $(i)$.
	\begin{align} \lambda r \int_0^t e^x \int_0^\infty e^{u\alpha  s} e^{-s} ds dx + \lambda (1- r) \int_0^t e^{u\alpha  s} ds 
		&\leq \frac{\lambda[1-(1-r)u \alpha]}{u\alpha(1-u \alpha)}e^{u\alpha t}. \end{align}
	
	For $\lambda=\frac{1}{(1+\sqrt{r})^2}$, we note that $\alpha=1/(1+\sqrt{r})$. Let $\mathcal{H}_1$ be the set of measurable function  $h: [0,\infty) \rightarrow [1,\infty)$ such that $h$ is non-decreasing and $\sup_{t\geq0} h(t)e^{-t/(1+\sqrt{r})}<\infty$. We define the following mapping from $\mathcal{H}_1$ to $\mathcal{H}_1$,
	$$ \Psi_1: h \rightarrow 1 + \frac{r}{(1+\sqrt{r})^2} \int_{0}^{t} e^x \int_{x}^{\infty} h(s) e^{-s} ds dx +  \frac{1-r}{(1+\sqrt{r})^2} \int_{0}^{t} h(s) ds.$$
	
	We denote by $f\preccurlyeq g$, the partial order of point-wise domination on  $ \mathcal{H}_u$ (or $\mathcal{H}_1$),  i.e.  $f(t)\leq g(t)$ for all $t\geq0$.
	
	\begin{lem} For $r\in (0,1]$ and $\lambda\in (0,\lambda_1]$.
		\label{lembor}
		\begin{itemize}
			\item [$(i)$] Let $1<u<\bar{\gamma}$. If $\lambda<\frac{1}{(1+\sqrt{r})^2}$ and $f\in \mathcal{H}_u$ such that $f\preccurlyeq\Psi_u(f)$. Then for all $t\geq0$,
			$$f(t) \leq  \frac{u\alpha(1-u\alpha)}{u \alpha[1+ \lambda(1-r)-u\alpha] - \lambda} e^{u\alpha t}.  $$ 
			\item [$(ii)$] If $\lambda=\frac{1}{(1+\sqrt{r})^2}$ and $f\in \mathcal{H}_1$ such that $f \preccurlyeq \Psi_1(f)$, then for all $t\geq0$,
			$$f(t) \leq C e^{t\slash (1+\sqrt{r})}, \hbox{ for some $C\geq1$}.$$
		\end{itemize}
	\end{lem}
	For the proof of this lemma we follow more closely the proof of  \cite[Lemma 5.2]{B2014} instead of \cite[Lemma 2]{B2008}.\\
	\begin{proof}
		We start with $(i)$. Let $h_0=f\in \mathcal{H}_u$ and for $k\geq 1$, we define $h_k=\Psi_u(h_{k-1})$.
		Since $f\in \mathcal{H}_u$, there exists $C_0>0$ such that $f\preccurlyeq g$, where $g(t)=C_0 e^{u\alpha t}$, for all $t\geq 0$. Applying $\Psi_u$ to $h_0$ gives that
		\begin{align*} h_1(t) &= 
			L e^{u\alpha t} + \frac{\lambda[1 - (1-r)u\alpha ]C_0e^{u\alpha t}}{u\alpha (1-u\alpha)} - \frac{\lambda[1 - (1-r)u\alpha ]C_0}{u\alpha (1-u \alpha)}
			\\
			& \leq L e^{u\alpha t} + \frac{\lambda[1-(1-r)u\alpha]}{u\alpha(1-u \alpha)}C_0e^{u\alpha t} = C_1 e^{u\alpha t}, 
		\end{align*}
		where $C_1 =L+ C_0 \frac{\lambda[1-(1-r)u\alpha]}{u\alpha (1-u\alpha)}$. By Lemma \ref{key}, it follows that, $0<\frac{\lambda[1-(1-r)u\alpha]}{u\alpha (1-u\alpha)}<1$.  \\
		By recurrence, $\lim \sup_k h_{k}(t)\leq C_{\infty} e^{u\alpha t}$ for all $t\geq 0$, where
		$C_{\infty}=\frac{L u\alpha (1-u\alpha)}{u\alpha[1+\lambda(1-r)-u\alpha] -\lambda}$. Again, by Lemma \ref{key}, it follows that  $C_{\infty}$ is positive and finite.\\
		We notice that the mapping $\Psi_u$ is monotone on $\mathcal{H}_u$ i.e. if $f_1,f_2\in \mathcal{H}_u$ are such that $f_1 \preccurlyeq f_2$, then $\Psi_u(f_1) \preccurlyeq \Psi_u(f_2)$. Since $\Psi_u(f) \preccurlyeq \Psi_u(g):=h_1(t)$, the assumption that $f\preccurlyeq \Psi_u(f)$ implies that $f(t)\leq C_{\infty} e^{u\alpha t}$ for all $t\geq 0$.\\ 
		Part $(ii)$ is simpler. For $r\in (0,r]$, we note that at $\lambda=\frac{1}{(1+\sqrt{r})^2}$, we have $\alpha=\frac{1}{(1+\sqrt{r})}$.  Since $f\in \mathcal{H}_1$,  there exists $C_0\geq 1$ such that $f\preccurlyeq g$, where $g(t)=C_0 e^{t/(1+\sqrt{r})}$, $t\geq0$. Applying $\Psi_1$ to $g$, we obtain that
		$$ \Psi_1(g)=(1-C_0) + C_0 e^{t/(1+\sqrt{r})}\leq C_0e^{t/(1+\sqrt{r})}.$$
		The mapping $\Psi_1$ is monotone on $\mathcal{H}_1$. Hence, $\Psi_1(f)\preccurlyeq \Psi_1(g)\preccurlyeq g$, and the assumption that $f\preccurlyeq \Psi_1(f)$ implies that  $f \preccurlyeq g$. 
	\end{proof}
	Now we can proceed with the proof of Proposition \ref{moments} $(i)$.
	\begin{proof}[Proof of Proposition \ref{moments} $(i)$]
		We define $f_1^{(n)}(t)=\mathbb{E}[\min\{ Y(t), n \}]$, $n\geq1$, from (\ref{RDE}) we obtain that
		\begin{align} \min \{ Y(t), n\} \stackrel{st}{\leq}  1 + \sum_{i: \xi_{i}^{(1)}\leq t} \min \{ Y_{i}(\xi_{i}^{(1)}+K_i), n\} +  \sum_{i: \xi_{i}^{(2)}\leq t} \min \{ Y_{i}(t-\xi_{i}^{(2)}),n \}, \label{RDE2}  \end{align}
		where `` $\stackrel{st}{\leq}$'' stands for stochastic dominance. Taking the expectation, we have that, for all $t\geq0$,
		\begin{align}\label{ff}
			f_1^{(n)}(t) \leq 1 + \lambda r \int_{0}^{t} e^x \int_{x}^{\infty}  f_1^{(n)}(s)e^{-s}ds dx +  \lambda (1-r)  \int_{0}^{t}  f_1^{(n)}(t-x) dx. 
		\end{align}
		
		Consequently, $f_1^{(n)}\preccurlyeq \Psi_u(f_1^{(n)})$. For $\lambda\in (0,\lambda_1)$, Lemma $\ref{lembor}$ $(i)$ implies that, for all $t\geq0$,
		$$ f_1^{(n)}(t)\leq \frac{u\alpha(1-u\alpha)}{u \alpha[1+ \lambda(1-r)-u\alpha] - \lambda} e^{u\alpha t}.  $$
		The Monotone Convergence Theorem implies that, for all $t\geq0$, $\lim_{n\rightarrow \infty} f_1^{(n)}(t)=f_1(t)$ exists and is bounded by $\frac{u\alpha(1-u\alpha)}{u \alpha[1+ \lambda(1-r)-u\alpha] - \lambda} e^{u\alpha t}$. Therefore, $f_1$ satisfies the differential equation $(\ref{B})$ and is equal to $x_a(t)=(a-1)e^{\alpha t}+ ae^{\beta t}$ for some constant $a\geq0$. Since $f_1(t)\leq \frac{u\alpha(1-u\alpha)}{u \alpha[1+ \lambda(1-r)-u\alpha] - \lambda} e^{u\alpha t} $, and $u\alpha<\beta$, it follows that $f_1(t)=x_0(t)=e^{\alpha t}$. Therefore, 
		$$\mathbb{E} [N]= \int^\infty_0 f_1(t) e^{-t} dt = \frac{1}{(1-\alpha)} = \frac{2}{1-\lambda(1-r) + \sqrt{\Delta}}. $$
		Analogously, at $\lambda=(1+\sqrt{r})^{-2}$, we obtain that $f_1(t)\leq C e^{t/(1+\sqrt{r})}$, for some $C\geq1$. Since $f_1(t)=x_a(t)=(at+1)e^{t/(1+\sqrt{r})}$, for some constant $a\geq0$, the upper bound on $f_1(t)$ implies that $f_1(t)=e^{t/{(1+\sqrt{r})}}$. Therefore, 
		$$\mathbb{E} [N]= \int^\infty_0 f_1(t) e^{-t} dt = \frac{1+\sqrt{r}}{\sqrt{r}}. $$ \end{proof}

\subsection{Model adding deleterious mutations}

\begin{proof}[Proof of Theorem \ref{TNSrp}]
Consider the process $\mathcal{B}(\lambda,r,p)$ introduced in Section~3 and let $\widehat{\mathcal{B}}(\lambda,r,p)$ be the subprocess obtained by discarding all deleterious mutants (sterile individuals). Since sterile pathogens produce no offspring, discarding them does not change the event of survival, hence $\mathcal{B}(\lambda,r,p)$ survives if and only if $\widehat{\mathcal{B}}(\lambda,r,p)$ survives.

Now observe that each pathogen in $\widehat{\mathcal{B}}(\lambda,r,p)$ gives birth at rate $\lambda$, but only births that are either copies (probability $1-r$) or beneficial mutants (probability $rp$) are retained in $\widehat{\mathcal{B}}(\lambda,r,p)$. By thinning of the Poisson process, the retained births form a Poisson process of rate
\[
\tilde{\lambda}=\lambda(1-r+rp),
\]
and, conditional on a retained birth, it is a beneficial mutation with probability
\[
\tilde{r}=\frac{rp}{1-r+rp}.
\]
Therefore the subprocess $\widehat{\mathcal{B}}(\lambda,r,p)$ has the same law as the original beneficial-mutation model $\mathcal{B}(\tilde{\lambda},\tilde{r})$ introduced in Section~2. Applying Theorem~\ref{TNS} we obtain that $\widehat{\mathcal{B}}(\lambda,r,p)$ dies out almost surely if and only if
\[
\tilde{\lambda}\le (1+\sqrt{\tilde{r}})^{-2}.
\]
A direct simplification yields
\[
\tilde{\lambda}(1+\sqrt{\tilde{r}})^2
= \lambda(1-r+rp)\left(1+\sqrt{\frac{rp}{1-r+rp}}\right)^2
= \lambda\left(\sqrt{1-r+rp}+\sqrt{rp}\right)^2.
\]
Hence the last inequality is equivalent to \eqref{eq:theta-closed}.

\end{proof}

\begin{proof}[Proof of Proposition \ref{prop:moments-B-lrp}]
We proceed as in the proof of Theorem~3.2, using the same thinning/coupling.
Set
\[
q:=1-r(1-p)=1-r+rp,\qquad \tilde\lambda:=\lambda q,\qquad \tilde r:=\frac{rp}{q}.
\]
In the coupling used in the proof of Theorem~3.2, the viable sub-process evolves exactly as the beneficial-mutation model $\mathcal B(\tilde\lambda,\tilde r)$.
Consequently, \emph{replacing} $(\lambda,r)$ by $(\tilde\lambda,\tilde r)$ in Proposition~2.3 yields $\mathbb{E}[N_{\mathrm v}]$
(in the finite-mean regime):
\begin{equation}\label{eq:ENv}
\mathbb{E}[N_{\mathrm v}]
=
\frac{2}{\,1-\tilde\lambda(1-\tilde r)+\sqrt{\bigl(1+\tilde\lambda(1-\tilde r)\bigr)^2-4\tilde\lambda}\,}.
\end{equation}
Using $\tilde\lambda(1-\tilde r)=\lambda(1-r)$ and $\tilde\lambda=\lambda(1-r+rp)$, we simplify \eqref{eq:ENv} to
\[
\mathbb{E}[N_{\mathrm v}]
=
\frac{2}{\,1-\lambda(1-r)+\sqrt{(1+\lambda(1-r))^2-4\lambda(1-r+rp)}\,}.
\]
Note that $N_{\mathrm v}$ does not count sterile mutants, so it remains to account for them.

Conditional on the entire viable sub-process,
births produced by viable individuals form a Poisson process of rate $\lambda$ and each such birth is declared
\emph{viable} with probability $q$ and \emph{sterile} with probability $1-q$, independently. Hence, conditional on the viable sub-process, the viable births and the
sterile births are obtained by a binomial thinning of the same underlying birth process, and therefore their
conditional expectations satisfy
\[
\mathbb{E}\!\left[S\,\middle|\,\mathcal F_{\mathrm v}\right]
=\frac{1-q}{q}\,\mathbb{E}\!\left[N_{\mathrm v}-1\,\middle|\,\mathcal F_{\mathrm v}\right],
\]
where $\mathcal F_{\mathrm v}$ is the $\sigma$-field generated by the viable sub-process (and $N_{\mathrm v}-1$ is the
number of viable births). Taking expectations gives
\[
\mathbb{E}[S]=\frac{1-q}{q}\,(\mathbb{E}[N_{\mathrm v}]-1).
\]
Therefore, \(\mathbb{E}[N^{(p)}]=\mathbb{E}[N_{\mathrm v}]+\mathbb{E}[S]\), and the expressions in Proposition \ref{prop:moments-B-lrp} follow from simple algebraic manipulations. Finally, if $\mathbb{E}[N_{\mathrm v}]=\infty$, then $\mathbb{E}[S]=\mathbb{E}[N^{(p)}]=\infty$. 
\end{proof}

%%%%%%%%%%%%%%%%%%%%%%%%%%%%%%%%%%%%%%%%%%%%%%%%%%%%%%%%%%%%%%%%%%%%%%%%%%%%%%%%%%%%%%%%%%%%%%%%%%%%%%%%%%%%%%%%%%%%%%%%%%%%%%%%%%%%%%%%%%%%%%

\backmatter

\bmhead{Acknowledgements}
Research supported by ANID-FONDECYT Iniciaci\'on grant (11230220), FAPESP (2023/13453-5) and Universidad de Antioquia (Project No. 2025-80410).

\end{document}